\documentclass[conference]{IEEEtran}
\IEEEoverridecommandlockouts
\usepackage{cite}
\usepackage{amsmath,amssymb,amsfonts}
\usepackage{graphicx}
\usepackage{xcolor}
\def\BibTeX{{\rm B\kern-.05em{\sc i\kern-.025em b}\kern-.08em
    T\kern-.1667em\lower.7ex\hbox{E}\kern-.125emX}}

\usepackage{amsmath,graphicx,hyperref}
\usepackage[caption=false,font=footnotesize]{subfig}
\usepackage{amsmath,graphicx,cite,amssymb}
\usepackage{amsfonts,multirow,bm,array,stfloats}
\usepackage{xcolor}
\usepackage{enumitem}
\usepackage[english]{babel}
\usepackage{amsthm}
\usepackage{xpatch}
\usepackage{graphicx,float,cite,amssymb}
\usepackage{graphics}
 \DeclareGraphicsExtensions{.pdf,.jpeg,.png,.jpg}   
\usepackage{color}
\usepackage{bbm}
\usepackage{comment}
\usepackage{mdframed}
\usepackage{algorithm,algpseudocode}
\usepackage{pifont}
\usepackage{balance}

\newcommand{\xopt}{x_{\star}}

\newcommand{\bk}{\mathbi{k}}
\newcommand{\bz}{\mathbi{z}}

\newcommand{\tb}{\mathbi{\Tilde{b}}}

\newcommand{\vecone}{\mathbf{1}}

\newcommand{\mathbi}[1]{\bm{#1}}

\newcommand{\bx}{\mathbf{x}}
\newcommand{\by}{\mathbf{y}}

\newcommand{\bb}{\mathbf{b}}

\newtheorem{proposition}{Proposition}
\newtheorem{lemma}{Lemma}

\newtheorem{remark}{Remark}

\newtheorem{assumption}{Assumption}
\usepackage{balance}

\usepackage{etoolbox}

\AtBeginEnvironment{equation}{\small}
\AtBeginEnvironment{equation*}{\small}
\AtBeginEnvironment{align}{\small}
\AtBeginEnvironment{align*}{\small}
\AtBeginEnvironment{gather}{\small}
\AtBeginEnvironment{gather*}{\small}
\AtBeginEnvironment{multline}{\small}
\AtBeginEnvironment{multline*}{\small}

\begin{document}

\title{Budget-Constrained Multi-Consensus\\ Decentralized Gradient Descent
\thanks{The work of S. Ren and E. G. Larsson was supported in part by ELLIIT, VR, and the KAW foundation. The work of N. Michelusi was funded in part by NSF under grant CNS-2129015.}}

\author{\IEEEauthorblockN{Shuyi Ren}
\IEEEauthorblockA{\textit{Dept. of Electrical Engineering (ISY)} \\
\textit{Link\"{o}ping University}\\
Sweden \\
shuyi.ren@liu.se}
\and
\IEEEauthorblockN{Nicol\`{o} Michelusi}
\IEEEauthorblockA{\textit{School of ECEE} \\
\textit{Arizona State University}\\
United States\\
nicolo.michelusi@asu.edu}
\and
\IEEEauthorblockN{Erik G. Larsson}
\IEEEauthorblockA{\textit{Dept. of Electrical Engineering (ISY)} \\
\textit{Link\"{o}ping University}\\
Sweden \\
erik.g.larsson@liu.se}
}

\maketitle

\begin{abstract}
We investigate decentralized gradient descent (DGD) with emphasis on efficient communication and computation resource utilization under budget constraints. As a first step toward the broader communication-computation allocation problem, we consider and analyze a \textit{multi-consensus decentralized gradient descent} (mcDGD) scheme, where the number of consensus rounds and the stepsize are allowed to vary across iterations. Building on a unified analytical framework for DGD, we derive finite-time convergence bounds that explicitly characterize the interaction between consensus quality and optimization dynamics. Our analysis requires only convexity of the local objective functions while assuming smoothness and strong convexity of the global objective. The resulting bounds enable a principled consensus-allocation strategy under resource constraints, for which we show that equal allocation of consensus rounds across iterations is optimal under our stepsize rule, up to integer rounding. Numerical experiments corroborate the theoretical findings and demonstrate favorable communication--computation tradeoffs compared with existing multi-consensus decentralized optimization baselines.
\end{abstract}

\begin{IEEEkeywords}
Decentralized optimization, convex optimization, communication-computation resource efficiency
\end{IEEEkeywords}

\section{Introduction}
\label{sec:intro}

In decentralized optimization, a network of agents collaboratively
minimizes a global objective using local computation and peer-to-peer
communication~\cite{yang2019survey,nedic2018multi}. In applications
such as decentralized machine learning, distributed inference and
estimation, and multi-agent
coordination~\cite{kar2013consensus+}, the global
objective is typically a sum of local objective
functions. Decentralized gradient descent (DGD) is a fundamental and
widely used algorithm for this class of
problems~\cite{tsitsiklis1984distributed,Nedic_Static2009,Duchi_2012,pu2021distributed,lian2017can,shi2015extra}.
It operates by alternating between consensus formation and local gradient updates.

Communication plays a crucial role in decentralized optimization. Since the communication effort required to achieve a desired consensus accuracy depends on the network connectivity, convergence performance is influenced by both the graph topology and the available communication and computation resources. While a substantial body of work has focused on improving convergence through topology and communication protocol design~\cite{koloskova2020unified,Asaad2024ExtendingGraph,liao4145292consensus,nedic2017achieving,shah2024adaptive,jakovetic2014fast,chen2012fast}, the question of how to allocate communication and computation resources under a fixed budget remains relatively less understood. This question is the focus of this paper.

Both communication and computation consume resources, and their relative costs depend on hardware capabilities, network conditions, synchronization latency, and the complexity of the local objectives. This motivates a two-stage resource-allocation perspective. First, for a fixed number $T$ of gradient updates and a fixed total number $K$ of consensus rounds, one can optimize how these $K$ rounds are allocated across the $T$ iterations. Then, given an overall resource budget, one can further optimize the choice of $T$ and $K$ by accounting for the respective costs of computation and communication.

In this work, we focus primarily on the first stage of this problem. We study a version of DGD where the number of consensus steps is varied across iterations, termed \textit{multi-consensus DGD} (mcDGD) herein. Leveraging the Mean Hessian Theorem and the contraction mapping framework in \cite{larsson2025unifiedDGD}, 
we develop a convergence analysis of mcDGD and a principled method for optimally allocating a prescribed consensus budget across the optimization horizon, for a specific choice of time-varying step sizes. Our analysis requires strong convexity of the \emph{global} objective, but only convexity of
the local objectives. The resulting allocation rule can subsequently be combined with the overall budget constraint when selecting the total numbers of gradient updates and consensus rounds.

Different aspects of mcDGD have been studied previously~\cite{berahas2018balancing,berahas2021convergence,choi2022convergence}. 
Reference~\cite{berahas2018balancing} proposed a multi-consensus version of standard DGD, termed DGD$^t$, as well as diffusion variants termed NEAR-DGD$^t$ and NEAR-DGD$^+$. 
The analysis in~\cite{berahas2018balancing} establishes convergence guarantees under strongly convex local objectives, and was extended in~\cite{berahas2021convergence} to include multiple local gradient updates per iteration. 
More recently, \cite{choi2022convergence} relaxed the requirement of strong convexity of the local objectives, requiring instead strong convexity only at the global objective level.

However, these works share two limitations from the perspective of resource-aware design. 
First, their bounds include stepsize-independent asymptotic error terms (see Theorem 5.3 of~\cite{berahas2018balancing}, Theorem III.6 of~\cite{berahas2021convergence}, and Theorem 2.1 of~\cite{choi2022convergence}), suggesting that the bounds may be conservative in the small-stepsize regime. 
Second, they analyze prescribed consensus-allocation policies, rather than optimizing the allocation of consensus rounds under a prescribed communication budget. 
In contrast, our analysis yields an asymptotic error that can be made arbitrarily small with a proper stepsize selection, consistently with known results on fixed-consensus DGD~\cite{larsson2025unifiedDGD}, and provides a principled method for budget-aware consensus allocation based on finite-time convergence bounds.

Different aspects of mcDGD have also been studied previously~\cite{berahas2018balancing,berahas2021convergence,choi2022convergence,michelusi26}. Specifically, \cite{berahas2018balancing} proposed a multi-consensus version of standard DGD, termed DGD$^t$, as well as diffusion variants termed NEAR-DGD$^t$ and NEAR-DGD$^+$. Their analysis establishes convergence guarantees for these algorithms for strongly convex local objectives. Reference~\cite{berahas2021convergence}, building upon \cite{berahas2018balancing}, further considered the use of multiple gradient updates per iteration. However, the resulting bounds include stepsize-independent asymptotic error terms (see Theorem 5.3 of \cite{berahas2018balancing} and Theorem III.6 of \cite{berahas2021convergence}), which suggests that these bounds may be conservative in the small-stepsize regime. In contrast, our analysis yields an asymptotic error that can be made arbitrarily small by the proposed stepsize rule, consistently with known results on fixed-consensus DGD in~\cite{larsson2025unifiedDGD}. More importantly, the aforementioned works study the convergence behavior of prescribed consensus-allocation policies, whereas we explicitly account for the communication and computation costs of multiple consensus rounds and optimize their allocation under a prescribed resource budget.

More recently, \cite{choi2022convergence} relaxed the strong convexity requirement on the local objectives. Nevertheless, similar to~\cite{berahas2018balancing,berahas2021convergence}, their analysis still leads to a non-vanishing asymptotic error floor (see Theorem 2.1 of \cite{choi2022convergence}). Moreover, the methods in \cite{berahas2018balancing,berahas2021convergence,choi2022convergence} rely on heuristic choices of the number of consensus rounds per iteration, whereas we provide a principled and systematic way of making this choice, based on finite-time convergence bounds. Budget-aware design of DGD under communication and gradient noise has recently been studied in~\cite{michelusi26}, where the budget complexity of reaching a target accuracy is analyzed. However, that work considers the conventional DGD protocol, whereas here we focus on budget-aware allocation in the multi-consensus setting.

The experiments with our proposed methods in Section~\ref{sec:simulation} illustrate the proposed communication allocation mechanism under a fixed overall resource budget. By varying the available budget and the corresponding numbers of computation and communication operations, we demonstrate how the proposed allocation strategy makes effective use of the available resources.

\textbf{Contributions:} Our specific technical contributions are: 1) a new analytical framework and convergence bound for mcDGD, leveraging the tools from \cite{larsson2025unifiedDGD}; 2) a principled method for optimally allocating a fixed number of consensus rounds
across gradient iterations based on this bound. 
This provides a first stage toward a general budget-constrained design, in which the total numbers of gradient updates and consensus rounds can subsequently be selected according to the available computation and communication budget. Simulation results illustrate our results and demonstrate clear performance gains compared to the method
in \cite{berahas2018balancing}. While our analysis currently assumes a single gradient evaluation per iteration, it represents a first step toward solving the more general problem with arbitrary numbers of consensus and local gradient updates per iteration.

\section{Multi-consensus DGD}

\label{sec:tv_DGD}

\subsection{System Model}

We consider a network of $N$ agents, connected through a communication network characterized by a weight matrix $W$: $[W]_{i,j}\neq 0$ if node $j$ can communicate with node $i$, and $[W]_{i,j}=0$ otherwise. Let $f_n:\mathbb{R}^M\rightarrow\mathbb{R}$ be the local objective function of the $n$th agent; for example, it represents the loss on the local dataset of device $n$. For notational simplicity, we take $M=1$ in the following analysis; all results extend immediately to $M>1$ at the expense of some additional notation,  representing the collection of agent variables as an $M\times N$ matrix or a stacked vector in $\mathbb{R}^{MN}$. 

The goal is to solve the global optimization problem 

\begin{equation}
    F(x):=\frac{1}{N}\sum_{n=1}^N f_n(x) \label{eq:global_obj}
\end{equation}

To solve \eqref{eq:global_obj} in a decentralized manner, each agent maintains a local copy of the optimization variable and exchanges information only with its neighbors. Standard DGD, as the most fundamental method in consensus optimization, works by having each agent $n$ forming a sequence of iterates. We let $\bx_t=[x_1^{(t)},\ldots,x_N^{(t)}]^\top$ be the $N$-vector whose $n$th component $x_n^{(t)}$ represents the variable value of the $n$th agent at iteration $t$. The iterates are formed by consensus averaging among the agents' neighbors using the weights in $W$, followed by a local gradient descent step:

\begin{equation}\label{eq:DGD_update}
    \bx_{t+1}=W\bx_t-\eta\nabla f(\bx_t),
\end{equation}
where $\eta$ is a positive stepsize.

From a communication and computation resource-efficiency perspective, however, performing exactly one consensus round per iteration may not always be desirable. 
This motivates a multi-consensus variant of DGD that can adapt the amount of communication performed across the optimization horizon.

\subsection{Multi-consensus DGD}

Specifically, in the proposed mcDGD framework, iteration $t$ performs $k_t$ consensus rounds before the local gradient update, leading to the update rule

\begin{equation}
\bx_{t+1} = W^{k_t}\bx_t - \eta_t \nabla f(\bx_t),
\end{equation}
where $W^{k_t}$ is the $k_t$-th power of $W$, representing $k_t$ consecutive consensus rounds at iteration $t$. The sequence $\{k_t \mid t=0,\ldots,T-1\}$ is optimized with the goal of improving the efficiency of mcDGD. To quantify efficiency, we adopt a cost model in which each gradient update incurs a computation cost $c_{\text{comp}}$ and each consensus step incurs a communication cost $c_{\text{comm}}$. Hence, running $T$ iterations with $k_t$ consensus rounds at iteration $t$ leads to a total cost of $T c_{\text{comp}} + \sum_{t=0}^{T-1} k_t c_{\text{comm}}$. Our convergence bound then enables a principled choice of $\{k_t\}$ that minimizes the convergence error under a given overall resource budget. In the next section, we derive this bound and use it to determine the optimal $\{k_t\}$.

\subsection{Notation and assumptions}

We use $\bx^\star=x^\star\vecone$ to denote the stacked vector whose entries are identical and equal to $x^\star$ and $x^\star$ is the minimizer of the global objective~\ref{eq:global_obj}, at which $F'(x^\star)=0$. Here $W\bx^\star=\bx^\star$, $\nabla f(\bx)=[f_1'(x_1),\ldots,f_N'(x_N)]^\top$ and $\vecone^\top\nabla f(\bx^\star)=\sum_{n=1}^Nf_n'(x^\star)=NF'(x^\star)=0$, but $\nabla f(\bx^\star)\neq 0$ in general. Moreover, all mathematical notation is standard. $\mathcal{O}(\cdot)$ denotes  big-O and $\varTheta(\cdot)$ denotes big-$\Theta$. $\rho(A)$ denotes the spectral radius of matrix $A$. $\vecone$ is the vector of all ones.

We make the following assumptions throughout this paper.
\begin{assumption}\label{ass:network}
 Assumptions on the network: 
\begin{enumerate}
    \item The network is connected and undirected, $W = W^\top$.
    \item $W$ satisfies $W\vecone = \vecone$.
    \item The network is not bipartite (i.e., the connectivity graph contains at least one odd cycle).
\end{enumerate}
\end{assumption}
Let $\{\lambda_1,\lambda_2, \ldots,\lambda_N\}$ be the  eigenvalues of $W$.
Given Assumption \ref{ass:network}, $\{ \lambda_n \}$ are real-valued and can be ordered as $1 = \lambda_1 > \lambda_2 \geq \cdots \geq \lambda_N > -1$. Note that we do not require $W$ to be non-negative or positive (semi-) definite. 
We will use $r_t$ to denote the spectral radius of the disagreement matrix in the $t$th iteration: $r_t:=\rho\!\left(W^{k_t}-\frac{\vecone\vecone^\top}{N}\right)$. 

\begin{assumption}\label{ass:convex_smooth}
Assumptions on the objective functions:
\begin{enumerate}
    \item Each local objective function $f_n(\cdot)$ is convex and $L$-smooth.
    \item The global objective $F(\cdot)$ is $L$-smooth and $\mu$-strongly convex.
\end{enumerate}
\end{assumption}
For instance, in an (overdetermined) least-squares problem, each agent may only observe one projection of the decision variable, so its local function is convex but flat in some directions. When the set of projections across all agents spans the whole space, the aggregate global objective becomes strongly convex. 

\section{Convergence Analysis of mcDGD}
\label{sec:convergence}

To analyze the convergence behavior of mcDGD, we decompose the error between the network-wide iterate $\bx_t$ and the global optimum $x^\star\vecone$ at iteration $t$

\begin{align*} 
    \bx_t-x^\star\vecone=\mathbi{\delta}_t+e_t\vecone,
\end{align*}
where we defined  the consensus error $\mathbi{\delta}_t$ and optimization error $e_t$ as

\begin{align*} 
    \mathbi{\delta}_t := \bx_t-\bar{x}_t\mathbf{1}:=\Pi_\vecone^\perp \bx_t,\quad 
    e_t := \bar{x}_t - x^\star.
\end{align*}
Here,  $\bar{x}_t = \frac{1}{N} \mathbf{1}^\top \bx_t$ is the network average and $\Pi_\vecone^\perp = I-\frac{\vecone\vecone^\top}{N}$ is the projection onto the orthogonal complement of $\mathrm{span}\{\vecone\}$.
Then, by orthogonality of the decomposition,

\begin{equation*} 
    \|\bx_t-\xopt\mathbf{1}\|\leq \sqrt{\|\mathbi{\delta}_t\|^2+N|e_t|^2}.
\end{equation*}

The following lemma quantifies how the errors evolve over the iterations.

\begin{lemma}[Error Dynamics]\label{lem:error_dynamics}
Define $E_t=|e_t|$ and $D_t=\|\mathbi{\delta}_t\|/\sqrt{N}$ for $t\ge 0$, and let $\by_t:=[E_t,\;D_t]^\top$. Under Assumption~\ref{ass:network}-\ref{ass:convex_smooth} given in Section \ref{sec:tv_DGD} and the stepsize condition $\eta_t\leq 2/(L+\mu)$, we have

\begin{equation}\label{eq:linear_system_tv_lem}
    \by_{t+1} \leq A_t\by_t+\bb_t,
\end{equation}
where $\bb_t=[0,\; \eta_t \|\nabla f(x^\star \mathbf{1})\|/\sqrt{N}]^\top$,

\begin{equation*}
A_t = \left[\begin{array}{cc}
1-\eta_t\mu, &\eta_t L \\
\eta_t L,& r_t+\eta_t L
\end{array}\right].
\end{equation*}
($\le$ denotes component-wise inequality.)
\end{lemma}

\begin{proof}
We can express the decomposed errors as 
\begin{align*} 
    e_{t+1} &=\bar{x}_{t+1}-x^\star, \\
    &\overset{(a)}{=}\frac{1}{N} \mathbf{1}^\top[\bx_t- \eta_t\nabla f(\bx_t)] - x^\star-\eta_t F'(\bar{x}_t)+\frac{\eta_t}{N}\vecone^\top\nabla f(\bar{x}_t\vecone),\\
    & =e_t-\eta_t F'(\bar{x}_t)- \frac{\eta_t }{N} \mathbf{1}^\top[\nabla f(\bx_t)-\nabla f(\bar{x}_t\vecone)],\\
    \mathbi{\delta}_{t+1} &=\bx_{t+1}-\bar{x}_{t+1}\vecone\\
    &\overset{(b)}{=}W^{k_t}\bx_t - \eta_t \nabla f(\bx_t)-\frac{1}{N} \mathbf{1}^\top[\bx_t - \eta_t\nabla f(\bx_t)]\vecone,\\
    &=W^{k_t}\mathbi{\delta}_t-\eta_t \Pi_\vecone^\perp\nabla f(\bx_t)+\eta_t \Pi_\vecone^\perp\nabla f(\bar{x}_t\vecone)-\eta_t \Pi_\vecone^\perp\nabla f(\bar{x}_t\vecone),\\
    &\qquad +\eta_t \Pi_\vecone^\perp\nabla f(x^\star\vecone)-\eta_t \Pi_\vecone^\perp\nabla f(x^\star\vecone)\\
    &=W^{k_t}\mathbi{\delta}_t-\eta_t \Pi_\vecone^\perp[\nabla f(\bx_t)-\nabla f(\bar{x}_t\vecone)],\\
    &\qquad -\eta_t \Pi_\vecone^\perp[\nabla f(\bar{x}_t\vecone)-\nabla f(x^\star\vecone)]-\eta_t \Pi_\vecone^\perp\nabla f(x^\star\vecone).
\end{align*}
where in $(a)$ and $(b)$ we used that $ F'(\bar{x}_t)=\vecone^\top\nabla f(\bar{x}_t\vecone)/N$ and $\vecone^\top W^{k_t}=\vecone^\top$. 

We then compute norms and bound them, using the triangle inequality, the fact that $\mathbi{\delta}_t$ lies in the orthogonal complement of $\vecone$ (i.e., $\Pi_\vecone^\perp \mathbi{\delta}_t = \mathbi{\delta}_t$), and the spectral properties of $W^{k_t}$:

\begin{equation}
    |e_{t+1}|\leq |e_t-\eta_t F'(\bar{x}_t)|+\frac{\eta_t}{\sqrt{N}}\|\nabla f(\bx_t)-\nabla f(\bar{x}_t\vecone)\|,
\end{equation}
\begin{align}
    \|\mathbi{\delta}_{t+1}\|\leq &\rho(W^{k_t}-\frac{\vecone\vecone^\top}{N})\|\mathbi{\delta}_t\|+\eta_t \|\nabla f(\bx_t)-\nabla f(\bar{x}_t\vecone)\|\nonumber\\
    &+\eta_t \|\nabla f(\bar{x}_t\vecone)-\nabla f(x^\star\vecone)\|+\eta_t \|\nabla f(x^\star\vecone)\|.
\end{align}

From smoothness,
\begin{equation*} 
\|\nabla f(\bx_t)-\nabla f(\bar{x}_t\vecone)\|\leq L\|\bx_t-\bar{x}_t\vecone\|=L\|\mathbi{\delta}_t\|,
\end{equation*}
\begin{equation*} 
\|\nabla f(\bar{x}_t\vecone)-\nabla f(x^\star\vecone)\|\leq L\|\bar{x}_t\vecone-x^\star\vecone\|=L\sqrt{N}|e_t|.
\end{equation*}

Under the stepsize condition $\eta_t\leq 2/(L+\mu)$, the following  tight   bound can be obtained using the MHT and contraction mapping theory (see Sections~III and IV-C in \cite{larsson2025unifiedDGD}):

\begin{align*} 
|e_t-\eta_t F'(\bar{x}_t)|
&=|\bar{x}_t-x^\star-\eta_t[F'(\bar{x}_t)-F'(x^\star)]|\\
&\leq(1-\eta_t\mu)|\bar{x}_t-x^\star|=(1-\eta_t\mu)|e_t|.
\end{align*}

Therefore,
\begin{equation*} 
|e_{t+1}|\leq(1-\eta_t\mu)|e_t|+\frac{1}{\sqrt{N}}\eta_t L\|\mathbi{\delta}_t\|,
\end{equation*}
\begin{equation*} 
\|\mathbi{\delta}_{t+1}\|\leq(r_t+\eta_t L)\|\mathbi{\delta}_t\|+\eta_t L\sqrt{N}|e_t|+\eta_t\|\nabla f(x^\star\vecone)\|,
\end{equation*}
where $r_t=\rho(W^{k_t}-\frac{\vecone\vecone^\top}{N})$. 
Dividing the second inequality by $\sqrt{N}$ and using $E_t=|e_t|$, $D_t=\|\mathbi{\delta}_t\|/\sqrt{N}$ yields \eqref{eq:linear_system_tv_lem}.
\end{proof}

Lemma~\ref{lem:error_dynamics} reduces the convergence analysis of mcDGD to the stability analysis of the two-dimensional time-varying linear system in \eqref{eq:linear_system_tv_lem}. The matrix $A_t$ captures the interaction between the optimization error and the consensus error: the diagonal terms describe their individual contractions, while the off-diagonal terms quantify their coupling. To obtain a convergent error bound, it is sufficient to choose the stepsize so that $A_t$ is contractive at every iteration.

A convenient sufficient condition is

\begin{equation}\label{eq:eta_t_condition}
    \eta_t < \frac{1-r_t}{L(1+L/\mu)},
\end{equation}
under which the eigenvalues of $A_t$ lie strictly inside the unit circle. This condition also reveals the coupling between optimization and communication: when the network mixing is poor, i.e., $r_t$ is close to one, the stepsize must be smaller; when the network mixing improves, a larger stepsize can be used.

Motivated by this observation, we choose an iteration-varying stepsize of the form

\begin{equation}\label{eq:eta_t}
    \eta_t=\kappa(1-r_t),
\end{equation}
where
\[
    0<\kappa<\frac{1}{L(1+L/\mu)}.
\]
With this choice, the coupling matrix becomes
\[
A_t =
I-(1-r_t)
\left[\begin{array}{cc}
\kappa\mu & -\kappa L \\
-\kappa L & 1-\kappa L
\end{array}\right]
:= I-(1-r_t)B .
\]
The advantage of this choice of stepsize is that all time variation is isolated in the scalar factor $1-r_t$, while the matrix $B$ is fixed across iterations. Moreover, the above condition on $\kappa$ ensures that $B$ is positive definite. Let $\xi_1$ denote the smallest eigenvalue of $B$
, i.e.,
\[
0<\xi_1\leq \xi_2<1.
\]
Then, we can write the eigen-decomposition of $B$ as
\[
    B=Q\Lambda Q^\top,
\]
where $\Lambda=\mathrm{diag}(\xi_1,\xi_2)$ and $Q$ is orthonormal. Consequently, all matrices $A_t$ share the same eigenvectors, and products of the form $\prod_t A_t$ can be characterized explicitly through the scalar factors $1-(1-r_t)\xi_i$, $i=1,2$. This allows us to explicitly characterize the accumulated error propagation over time. The original coupled two-dimensional error dynamics is decoupled into scalar product terms whose dependence on  $\{r_t\}$ becomes explicit. Combining this observation with  \eqref{eq:linear_system_tv_lem} yields the following upper bound on the error.

\begin{proposition}[Upper bound on the convergence error]
\label{prop:upper_bound_only}The norm of
$\by_T$ defined in Lemma~\ref{lem:error_dynamics} satisfies
\begin{align}
    \|\by_T\|\leq
    &\prod_{t=0}^{T-1}\left(1-(1-r_t)\xi_1\right)\left\|\by_0-B^{-1}\tb\right\|\nonumber\\
    &+\frac{C_\kappa\sqrt{L^2+\mu^2}}{\mu-\kappa(\mu L+L^2)},\label{eq:y_bar}
\end{align}
where
\[
    C_\kappa:=\frac{\kappa\|\nabla f(x^\star\mathbf{1})\|}{\sqrt{N}},\qquad\tb:=
    \begin{bmatrix}
    0\\
    C_\kappa
    \end{bmatrix}.
\]
\end{proposition}

\begin{proof}
We initialize $\bar{\by}_0=\by_0$ and define $\bar{\by}_t$ as a componentwise upper bound on $\by_t$ satisfying
\[
\bar{\by}_{t+1}=\left(I-(1-r_t)B\right)\bar{\by}_t+(1-r_t)\tb.
\]
By Lemma~\ref{lem:error_dynamics}, we have
\[
\by_t\leq\bar{\by}_t,
\]
element wise and, since both vectors consist of nonnegative error quantities,
\[
\|\by_t\|\leq\|\bar{\by}_t\|.
\]
Define the shifted variable
\[
\bz_t:=\bar{\by}_t-B^{-1}\tb.
\]
Then
\begin{align*}
    \bz_{t+1}
    &=\bar{\by}_{t+1}-B^{-1}\tb\\
    &=\left(I-(1-r_t)B\right)\bar{\by}_t+(1-r_t)\tb-B^{-1}\tb\\
    &=\left(I-(1-r_t)B\right)\left(\bar{\by}_t-B^{-1}\tb\right)\\
    &=\left(I-(1-r_t)B\right)\bz_t.
\end{align*}
Therefore,
\[
\bz_T=\prod_{t=0}^{T-1}\left(I-(1-r_t)B\right)\bz_0.
\]

Since $B$ is symmetric, it admits an eigendecomposition
\[
B=Q\Lambda Q^\top,\qquad\Lambda=\operatorname{diag}(\xi_1,\xi_2),
\]
where $Q$ is orthogonal. Hence,
\begin{align*}
    \prod_{t=0}^{T-1}\left(I-(1-r_t)B\right)
    &=Q\left[\prod_{t=0}^{T-1}\left(I-(1-r_t)\Lambda\right)\right]Q^\top.
\end{align*}
It follows that
\begin{align*}
    \left\|\prod_{t=0}^{T-1}\left(I-(1-r_t)B\right)\right\|_2
    &=\max_{i\in\{1,2\}}\prod_{t=0}^{T-1}\left(1-(1-r_t)\xi_i\right).
\end{align*}

Under the condition on $\kappa$, the eigenvalues of $B$ satisfy
\[
0<\xi_1\leq \xi_2<1.
\]
Moreover, since 
$r_t\in[0,1)$, $0< 1-r_t\leq 1$, and therefore
\[
0<1-(1-r_t)\xi_2\leq 1-(1-r_t)\xi_1\leq 1.
\]
Thus, the smallest eigenvalue, $\xi_1$, determines the slowest
contraction, and
\[
\left\|\prod_{t=0}^{T-1}\left(I-(1-r_t)B\right)\right\|_2=\prod_{t=0}^{T-1}\left(1-(1-r_t)\xi_1\right).
\]
Consequently,
\[
\|\bz_T\|\leq\prod_{t=0}^{T-1}\left(1-(1-r_t)\xi_1\right)\|\bz_0\|.
\]

Recalling that
\[
\bz_t=\bar{\by}_t-B^{-1}\tb,
\]
we have
\[
\bar{\by}_T=\bz_T+B^{-1}\tb.
\]
Hence, by the triangle inequality,
\begin{align*}
    \|\bar{\by}_T\|
    &\leq\|\bz_T\|+\|B^{-1}\tb\|\\
    &\leq\prod_{t=0}^{T-1}\left(1-(1-r_t)\xi_1\right)\left\|\bar{\by}_0-B^{-1}\tb\right\|+\|B^{-1}\tb\|,
\end{align*}
after noting that $\bar{\by}_0=\by_0$. It remains to evaluate the last term.

For
\[
B=
\begin{bmatrix}
\kappa\mu & -\kappa L\\
-\kappa L & 1-\kappa L
\end{bmatrix},
\]
we have
\[
B^{-1}\tb=\frac{C_\kappa}{\mu-\kappa(\mu L+L^2)}
\begin{bmatrix}
L\\
\mu
\end{bmatrix}.
\]
Therefore,
\[
\|B^{-1}\tb\|=\frac{C_\kappa\sqrt{L^2+\mu^2}}{\mu-\kappa(\mu L+L^2)}.
\]
Combining the above inequalities with
\[
\|\by_T\|\leq\|\bar{\by}_T\|
\]
yields \eqref{eq:y_bar}.

\end{proof}

Note that
\[
\frac{\kappa}{\mu-\kappa(\mu L+L^2)}
=\frac{\kappa}{\mu}\frac{1}{1-\frac{\kappa(\mu L+L^2)}{\mu}}
=\frac{\kappa}{\mu}+\mathcal{O}(\kappa^2).
\]
Hence,
\[
\frac{\kappa}{\mu-\kappa(\mu L+L^2)}=\varTheta(\kappa),\qquad\kappa\rightarrow 0^+.
\]
Therefore, the error floor
\[
\frac{\kappa\sqrt{L^2+\mu^2}}{\mu-\kappa(\mu L+L^2)}\frac{\|\nabla f(x^\star\mathbf{1})\|}{\sqrt{N}}
\]
is of order $\mathcal{O}(\kappa)$ and can thus be reduced by choosing a smaller $\kappa$. 

Moreover, \eqref{eq:y_bar} reveals a trade-off in the choice of $\kappa$. For sufficiently small $\kappa$, we have
\[
\xi_1=\kappa\mu+\mathcal{O}(\kappa^2),
\]
so increasing $\kappa$ strengthens the contraction factor
\[
1-(1-r_t)\xi_1
\]
and hence accelerates the decay of the transient term. At the same time, a larger $\kappa$ increases the error floor, which scales as $\mathcal{O}(\kappa)$.

Proposition~\ref{prop:upper_bound_only} holds for the stepsize choice in \eqref{eq:eta_t}. Extending the result to more general fixed stepsizes satisfying \eqref{eq:eta_t_condition} is left for future work.
Also, in the special case of a constant number of consensus rounds per iteration, say $k_t\equiv k$ for some
$k$, \eqref{eq:eta_t} gives the constant stepsize $\eta_t\equiv \eta$ for all $t$. 
Then Proposition \ref{prop:upper_bound_only} gives
\begin{equation}\label{eq:general_bound}
\limsup_{t\rightarrow\infty}
\|\bx_t-\xopt\mathbf{1}\|
= \mathcal{O}(\kappa) = 
\mathcal{O}(\eta),
\end{equation}
showing that DGD converges to a neighborhood of the optimum whose size scales linearly with the stepsize, consistently with classical DGD convergence results under constant stepsizes~\cite{Nedic_Static2009,Duchi_2012,jakovetic2014fast}.
This stepsize-dependent error floor differs from prior multi-consensus DGD analyses~\cite{berahas2018balancing,berahas2021convergence,choi2022convergence}, where the asymptotic error floor is determined by the network topology and the number of consensus rounds, and cannot in general be made arbitrarily small through stepsize selection.

Note how the convergence bound naturally separates transient and residual effects. The first term in \eqref{eq:upper_bound} determines how communication influences the error decay, while the second term characterizes the residual error induced by local objective heterogeneity.

\begin{remark}
Reference~\cite{choi2022convergence} also studies coupled consensus--optimization error dynamics under time-varying consensus mechanisms. However, its analysis assumes fixed stepsizes together with specific summability requirements on $\{r_t\}$, restricting the admissible consensus schedules. In contrast, our framework explicitly couples optimization and network mixing through~\eqref{eq:eta_t}, allowing iteration-varying stepsizes that naturally adapt to communication quality.
\end{remark}

\section{Optimizing the number of consensus rounds}

We next determine the sequence $\{k_t\}$ that minimizes the upper bound on the overall error, while satisfying a constraint on the total communication and computation cost. Proposition~\ref{prop:upper_bound_only} explicitly reveals how the communication effort affects convergence. By the definition of $r_t$, we have
\[
r_t
=
\nu^{k_t},
\qquad
\nu:=\max_{n\geq 2}|\lambda_n|\in[0,1).
\]
Substituting this relation into \eqref{eq:y_bar} gives
\begin{align}
\|\bx_T-\xopt\mathbf{1}\|
\leq&
\sqrt{N}\|\by_T\|
\nonumber\\
\leq&
\sqrt{N}
\prod_{t=0}^{T-1}
\left(
1-\xi_1+\nu^{k_t}\xi_1
\right)
\left\|
\by_0+B^{-1}\tb
\right\|
\nonumber\\
&+
\frac{
C_\kappa\sqrt{N(L^2+\mu^2)}
}{
\mu-\kappa(\mu L+L^2)
}.
\label{eq:upper_bound}
\end{align}

We observe that the communication schedule $\{k_t\}_{t=0}^{T-1}$
affects the convergence bound only through the product
\[
\prod_{t=0}^{T-1}
\left(
1-\xi_1+\nu^{k_t}\xi_1
\right).
\]
Since $\nu\in[0,1)$, increasing $k_t$ decreases $\nu^{k_t}$ and hence
reduces the corresponding contraction factor
\[
1-\xi_1+\nu^{k_t}\xi_1.
\]
Thus, allocating more consensus rounds tightens the transient part of the convergence bound, at the expense of increased communication cost.
Accordingly, for given $\kappa$ and $T$, designing the communication schedule reduces to minimizing this product term under a total resource budget. This leads to the optimization problem

\begin{subequations}
\label{prob_cost}
\begin{align}
\underset{\bk=\{k_t|t=0,\ldots,T-1\}}{\min}
&\quad
\sum_{t=0}^{T-1}
\ln(1-\xi_1+\nu^{k_t}\xi_1)
\label{prob_cost_obj}
\\
\text{subject to}
&\quad
Tc_{\mathrm{comp}}
+
\sum_{t=0}^{T-1}
k_t c_{\mathrm{comm}}
\le R.
\label{prob_cost_cons}
\end{align}
\end{subequations}

The following proposition shows that the optimal choice is to  make $\{k_t\}$ as close to constant as possible over the iterations.

\begin{proposition}[Optimality of equal communication resource allocation per iteration]\label{prop:optimality}
For problem \eqref{prob_cost}, under an optimal communication schedule, the number of consensus rounds $\{k_t\}_{t=0}^{T-1}$ take at most two adjacent integer values. Specifically, the optimal solution time-shares between

\[
\bar{k}
=
\left\lfloor
\frac{R/T-c_{\mathrm{comp}}}
{c_{\mathrm{comm}}}
\right\rfloor \quad \text{and} \quad \bar{k}+1.
\]
\end{proposition}

\begin{proof}
Let
\[
h(k):=\ln\left(1-\xi_1+\nu^k\xi_1\right).
\]
For $0<\xi_1<1$ and $0<\nu<1$, $h(k)$ is strictly decreasing and strictly convex on $\mathbb{R}$. In particular,
\[
h''(k)
=
\frac{
\xi_1(1-\xi_1)(\ln\nu)^2\nu^k
}{
\left(1-\xi_1+\xi_1\nu^k\right)^2
}
>0.
\]

Suppose, by contradiction, that an optimal communication schedule contains two entries $k_1$ and $k_2$ such that
\[
k_1-k_2\geq 2.
\]
Consider a new schedule obtained by replacing the pair $(k_1,k_2)$ with
\[
(k_1-1,k_2+1),
\]
while keeping all other entries unchanged. This modification preserves the total number of consensus rounds since
\[
(k_1-1)+(k_2+1)=k_1+k_2,
\]
and hence preserves the total resource consumption.

Let
\[
d:=k_1-k_2\geq 2,\qquad\theta:=\frac{1}{d}\in(0,1).
\]
Then
\[
k_1-1=(1-\theta)k_1+\theta k_2,
\]
and
\[
k_2+1=\theta k_1+(1-\theta)k_2.
\]
By the strict convexity of $h$, we have
\[
h(k_1-1)<(1-\theta)h(k_1)+\theta h(k_2),
\]
and
\[
h(k_2+1)<\theta h(k_1)+(1-\theta)h(k_2).
\]
Adding the two inequalities gives
\[
h(k_1-1)+h(k_2+1)<h(k_1)+h(k_2).
\]
Therefore, transferring one communication round from the larger entry to the smaller entry strictly decreases the objective while keeping the total communication cost unchanged, contradicting the optimality of the original schedule.

Hence, any optimal communication schedule must satisfy
\[
\max_t k_t-\min_t k_t\leq 1.
\]
In other words, the optimal values of $\{k_t\}_{t=0}^{T-1}$ can take at most two adjacent integer values.

Since $h(k)$ is decreasing, the available communication budget is fully utilized up to the integer constraint. Thus, the average number of communication rounds per iteration is
\[
\frac{R/T-c_{\mathrm{comp}}}{c_{\mathrm{comm}}},
\]
and the two possible integer values are
\[
\bar{k}=\left\lfloor\frac{R/T-c_{\mathrm{comp}}}{c_{\mathrm{comm}}}\right\rfloor
\]
and $\bar{k}+1$. Therefore, the optimal schedule time-shares between $\bar{k}$ and $\bar{k}+1$.
\end{proof}

In practice, finding the optimal $\{k_t\}$ is straightforward. First initialize $k_t=\bar{k}$ for all iterations. Then compute the remaining communication budget after this initialization.
Then distribute the residual budget by assigning one additional consensus step to as many iterations as permitted by the remaining resources. Note that the objective in \eqref{prob_cost_obj} is 
invariant to a reordering of the sequence $\{k_t\}$. Hence,  
the specific order in which the communication resources are used 
does not affect the worst-case optimality with respect to the bound in \eqref{prob_cost_obj}.

\section{Numerical Examples}
\label{sec:simulation}

\begin{figure*}[ht]
  \centering
  \subfloat[Star graph]{%
    \label{fig:star}%
    \begin{minipage}[t]{0.32\linewidth}
      \centering
      \includegraphics[width=\linewidth]{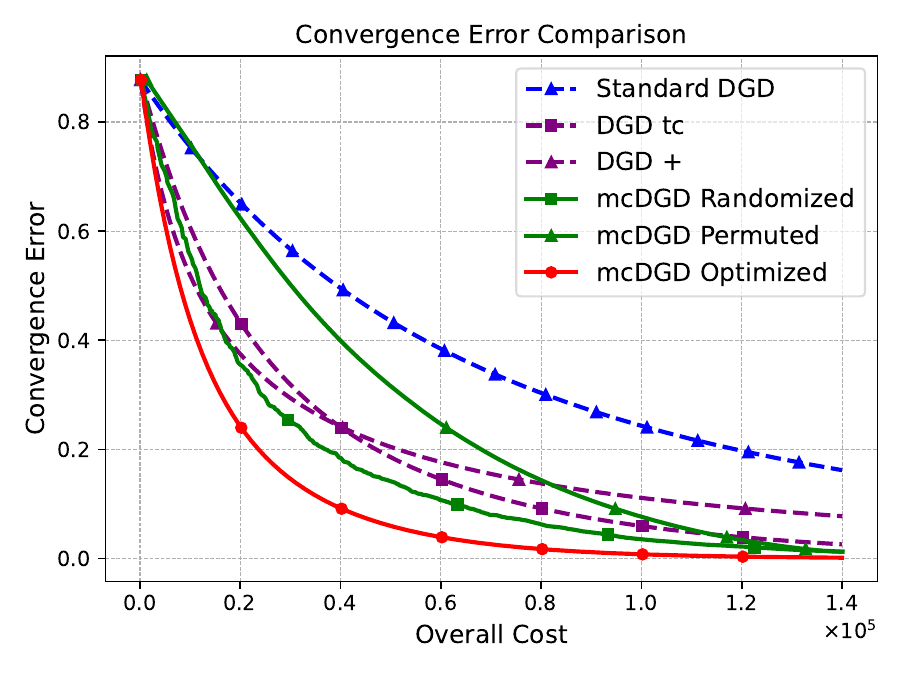}\\[0.6ex]
      \includegraphics[width=\linewidth]{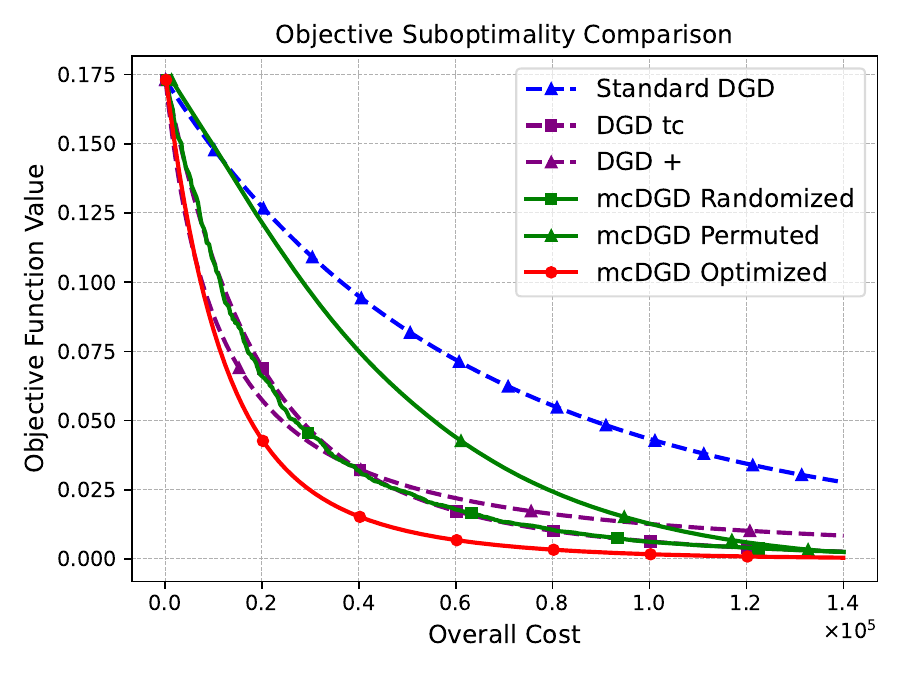}
    \end{minipage}
  }\hfill
  \subfloat[Slowed-down complete graph with $\gamma=0.1$]{%
    \label{fig:complete}%
    \begin{minipage}[t]{0.32\linewidth}
      \centering
      \includegraphics[width=\linewidth]{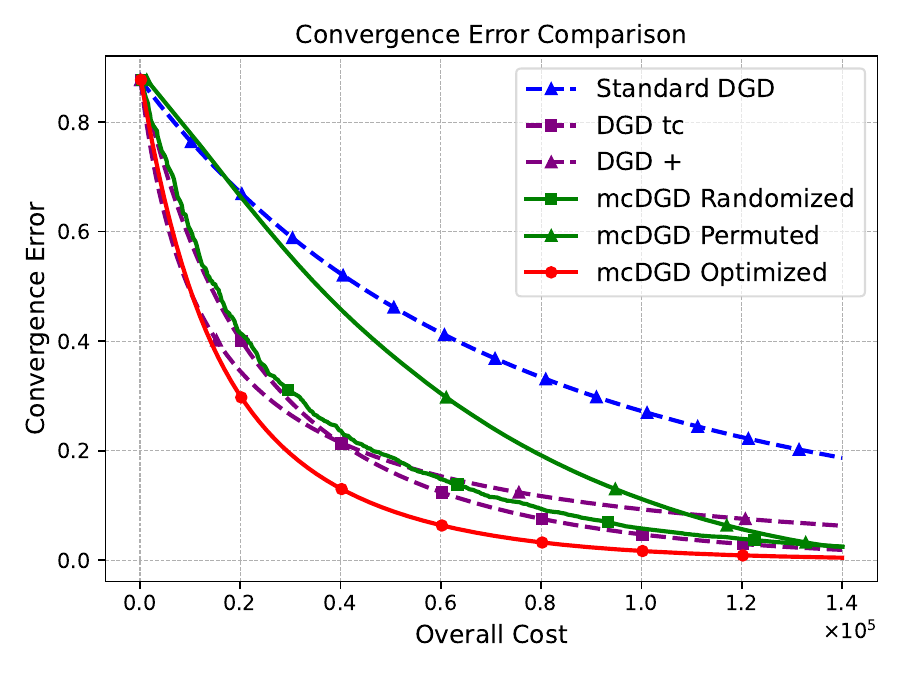}\\[0.6ex]
      \includegraphics[width=\linewidth]{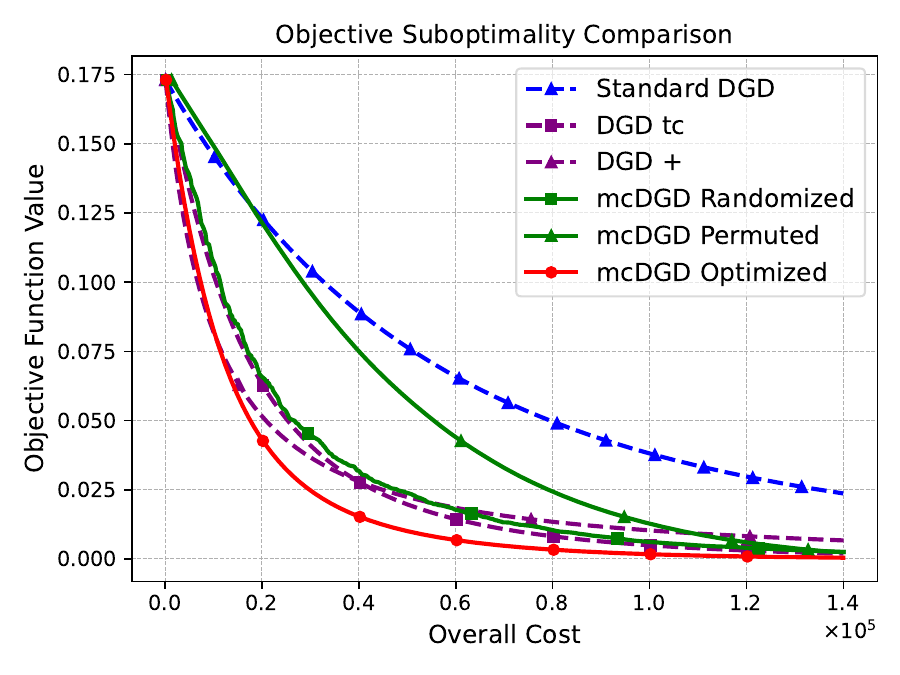}
    \end{minipage}
  }\hfill
  \subfloat[Ring graph]{%
    \label{fig:ring}%
    \begin{minipage}[t]{0.32\linewidth}
      \centering
      \includegraphics[width=\linewidth]{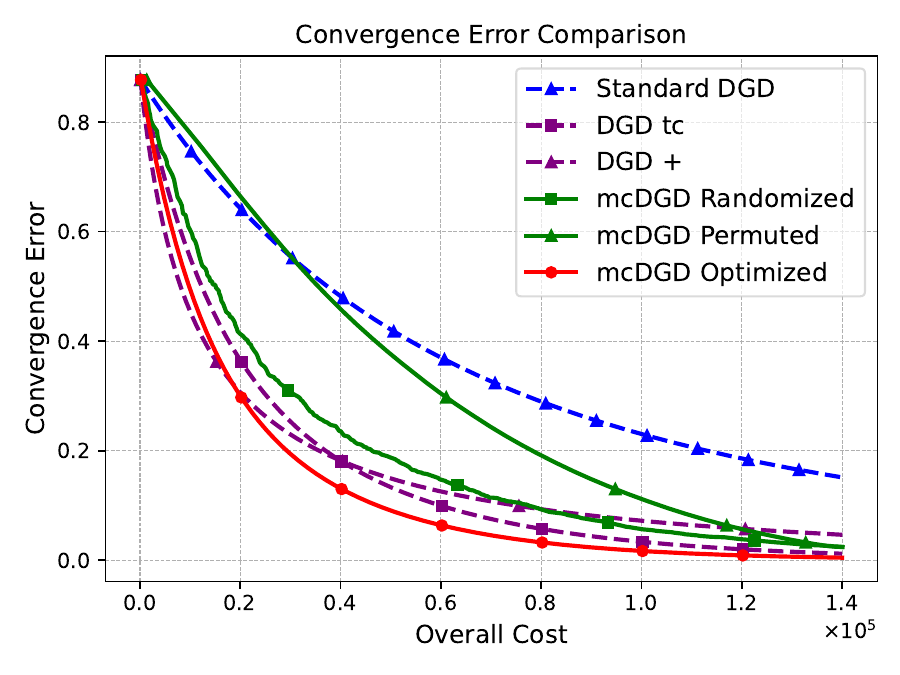}\\[0.6ex]
      \includegraphics[width=\linewidth]{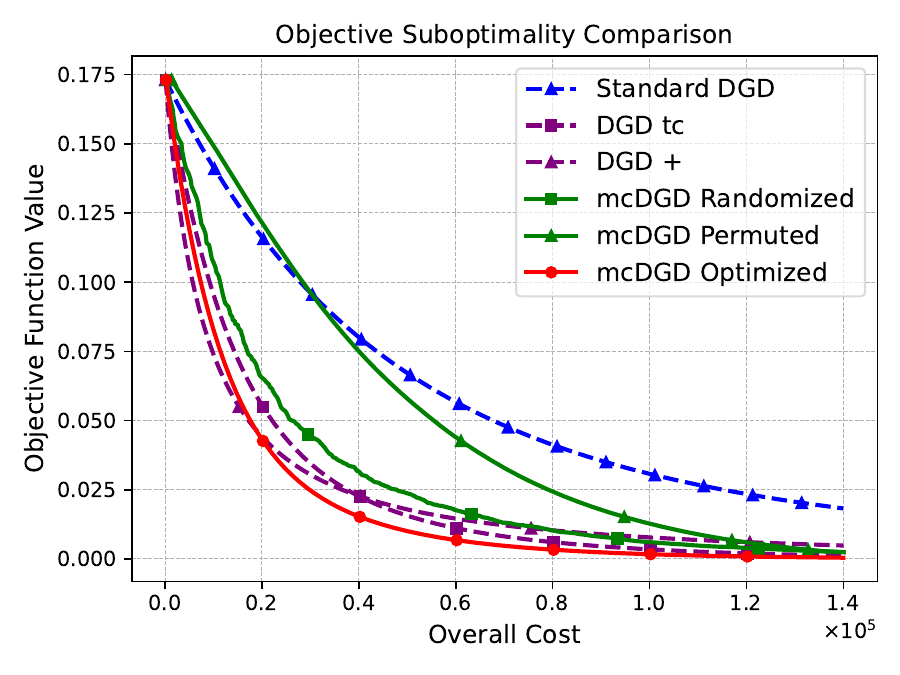}
    \end{minipage}
  } 
  \caption{Convergence error $\|\bx_t-x_\star\vecone\|^2$ and global objective function value after $T$ iterations $F(\bar{\bx}_T)$ as function of  overall consumed resources, which include both computation and communication. Parameters: $c_{\text{comp}}=100$ and $c_{\text{comm}}=1$.
  }
  \label{fig:simulation}
\end{figure*}

In this section, we consider decentralized least-squares regression. Each agent $n$ observes a scalar measurement $g_n\in\mathbb R$ and jointly optimizes an $M$-dimensional variable $\bx\in\mathbb R^M$. The local objective is
\[
f_n(\bx)
=
(g_n-G_n^\top\bx)^2,
\]
where $G_n\in\mathbb R^M$ is locally generated. We take $N=10$ agents and dimension $M=5$. In our implementation, the parameter $\kappa$ is selected from several candidate values within the admissible range
$\kappa<\frac{1}{L(1+L/\mu)}$, and we report the value that yields the best numerical performance among those tested. The selected $\kappa$ is then kept fixed across the considered mcDGD allocation strategies. Thus, $\kappa$ is empirically tuned rather than fixed a priori, although this numerical selection does not establish theoretical or global optimality over the entire admissible range. A joint theoretical optimization of $\kappa$ and the communication allocation $\{k_t\}$ is beyond the scope of the present work.

Importantly, the local objectives need not be strongly convex individually. Our analysis requires only that the aggregate global objective \eqref{eq:global_obj} is $\mu$-strongly convex. For the least-squares setup considered here, this holds whenever $N\ge M$ and the matrix formed by stacking all $\{G_n\}$ is full rank.

Figure~\ref{fig:simulation} compares performance for three network topologies:
\begin{itemize}
    \item a star graph,
    \item a slowed-down complete graph with
    
    \[
    W = (1-\gamma)I + \gamma \frac1N \mathbf1\mathbf1^\top,
    \]
    \item a ring graph.
\end{itemize}
For all methods, communication and computation costs are measured according to \eqref{prob_cost_cons}.

We compare the proposed mcDGD with the following benchmark methods:
\begin{itemize}
    \item ``\emph{Standard DGD}": decentralized gradient descent with one consensus step per iteration, $k_t=1,\forall t$. Following the stepsize design adopted in \eqref{eq:eta_t}, when $k_t=1$, the resulting constant stepsize becomes $\eta=(1-\nu)\kappa$.
    This choice ensures that standard DGD is evaluated under the same communication-aware stepsize scaling as the proposed mcDGD.
    
    \item ``\emph{DGD$^t$}" and ``\emph{DGD$^+$}": $\mathrm{DGD}^t$ denotes the multi-consensus extension of standard DGD proposed in~\cite{berahas2018balancing}, where a fixed number $t_c$ of consensus rounds is performed at every iteration, i.e., $k_t=t_c$. We additionally consider $\mathrm{DGD}^+$, a natural extension in which the number of consensus rounds increases with the iteration index, specifically $k_t=t$. This schedule is analogous to that used in $\mathrm{NEAR\mbox{-}DGD}^{+}$~\cite{berahas2018balancing}, but is applied here to the standard DGD update in \eqref{eq:DGD_update}. For both methods, we use a constant stepsize chosen to satisfy the largest possible stepsize that satisfies the condition
    \[
    \eta\le\min\!\left\{\frac{1+\lambda_N(W^{k_t})}{L},\frac{2}{L+\mu}\right\}
    \]
    at every iteration, following the condition derived for the standard-DGD-based $\mathrm{DGD}^t$ method in \cite{berahas2018balancing}.
\end{itemize}

For a fixed resource budget and the prescribed stepsize setting, our proposed mcDGD consistently achieves lower optimization error than fixed-consensus baselines after sufficiently many iterations, demonstrating the benefit of optimizing consensus allocation according to the resource model.

To exemplify the implications of Proposition~\ref{prop:optimality}, we additionally include two mcDGD variants using the same total communication budget $K\equiv\sum_{t=0}^{T-1}k_t$, defined as follows:

\begin{itemize}
\item random communication schedules generated by uniformly distributing a fixed total number of consensus rounds $K$ across the $T$ iterations\footnote{The random schedule is generated using the classical stars-and-bars construction~\cite{stanley2011enumerative}. We place $K$ indistinguishable consensus rounds (``stars") into $T$ iteration bins by sampling $T-1$ separator positions uniformly at random among $K+T-1$ locations. The resulting vector $(k_0,\ldots,k_{T-1})$ satisfies $\sum_t k_t=K$ and is uniformly distributed over all feasible allocations.}. Specifically, the sequence $\{k_t\}_{t=0}^{T-1}$ is sampled uniformly from the set

\[
\left\{
(k_0,\ldots,k_{T-1}) \in \mathbb{Z}_{\ge0}^{T}\mid
\sum_{t=0}^{T-1} k_t = K
\right\},
\] named as ``\emph{mcDGD Randomized}''.

\item permuted versions of the randomly generated $\{k_t\}$ sequence, named as ``\emph{mcDGD Permuted}".

\end{itemize}

The performance of these mcDGD variants are shown in green using square and triangle markers, respectively. 
They achieve identical final performance, confirming that the ordering of the consensus rounds does not affect the bound/objective in
 \eqref{prob_cost}. The randomly generated allocation of ${k_t}$ achieves a similar final accuracy, but exhibits inferior transient performance compared with the optimized allocation, illustrating the benefit of optimized communication allocation.

Finally, the proposed mcDGD with optimized $\{k_t\}$, shown in red, consistently achieves the best overall communication-computation efficiency across all considered network topologies.

\section{Conclusion and Future Work}
We developed an analytical framework for multi-consensus decentralized gradient descent (mcDGD) with iteration-varying stepsizes. The analysis requires smoothness and strong convexity of the global objective
but only convexity of the local objective functions. The resulting finite-time convergence bound explicitly characterizes the interplay between consensus quality and stepsize selection, enabling a principled communication resource allocation strategy under resource constraints. We showed that the optimal allocation assigns consensus rounds as evenly as possible across the iterations. Numerical experiments corroborate the theoretical findings and demonstrate consistent performance gains over other methods across different network topologies. 
Future work includes extending the analysis to 
allow iterations without consensus updates (i.e. $k_t=0$), as well as optimization of additional algorithm parameters such as $\kappa$ and $T$.

\balance
\bibliographystyle{IEEEbib}     
\bibliography{strings}

\color{red}

\end{document}